\documentclass[a4paper,12pt]{article}

\usepackage{epsfig}
\usepackage{amsfonts,eucal}
\usepackage{amssymb,amsmath,amsthm,color}
\usepackage[normalem]{ulem}

\usepackage[T2A]{fontenc}
\usepackage[cp1251]{inputenc}
\newtheorem{remark}{Remark}[section]
\newtheorem{theorem}{Theorem}[section]
\newtheorem{proposition}{Proposition}[section]
\newtheorem{lemma}{Lemma}[section]

\newcommand{\Ga}{\Gamma}
\newcommand{\ga}{\gamma}
\newcommand{\BB}{{\mathbb B}}
\newcommand{\XX}{{\mathbb X}}
\newcommand{\R}{{\mathbb R}}
\newcommand{\N}{{\mathbb N}}
\newcommand{\Z}{{\mathbb Z}}
\newcommand{\X}{{\R^d}}
\newcommand{\XXX}{{\mathfrak{X}}}
\newcommand{\La}{\Lambda}
\newcommand{\la}{\lambda}
\newcommand{\B}{\mathcal{B}}

\newcommand{\M}{\mathcal{M}}
\newcommand{\CB}{\mathcal B}

\begin{document}

\title{General contact processes: inhomogeneous models, models on graphs and on manifolds. }

\author{Sergey Pirogov\thanks{
Institute for Information Transmission Problems, Moscow, Russia
(s.a.pirogov@bk.ru).} \and Elena Zhizhina\thanks{Institute for
Information Transmission Problems, Moscow, Russia (ejj@iitp.ru).} }

\date{}

\maketitle

\begin{abstract}
The contact process is a particular case of birth-and-death processes on infinite particle configurations. We consider the contact models on locally compact separable metric spaces.
We prove the existence of a one-parameter set of invariant measures in the critical regime under the condition imposed on the associated Markov jump process. This condition, roughly speaking, requires the separation of any pair of trajectories of this jump process. The general scheme can be applied to the contact process on the lattice in a heterogeneous and random environments as well as to the contact process on graphs and on manifolds.
\\

Keywords: birth and death process, infinite particle configurations, critical regime, correlation functions, hierarchical equations
\end{abstract}

\section{Introduction}


Starting from the pioneer papers of Harris \cite{H}, Holley and Liggett \cite{HL} the contact process has become one of the most widely used population dynamics model, see also the monograph of Liggett \cite{Lig1985}.
While in most of the works the contact processes were considered on the lattice $\Z^d$, much of the interest in the recent years has focused on studying the contact processes running in continuous spaces, see e.g. \cite{FKO, KKP, KS}.
Contact processes are a particular case of continuous time birth and death processes on infinite particle configurations, and one of the basic problems concerning a contact process is to determine a stationary regime and to prove the existence of stationary measures.
In the mathematical literature, the birth and death rates of the contact processes are usually taken to be homogeneous (in space); therefore, the corresponding stationary measures on configurations are translation invariant.
Since homogeneous models do not quite accurately reflect reality due to heterogeneity in biological or social populations, contact processes in general spaces as well as in heterogeneous and random environments are of great importance for a better understanding of real-world networks.

One of the main features of the contact process is the clustering of the system, i.e. particles are grouped into large clouds of high density, which are located at large distances from each other.  It is worth noting that the appearance of a limiting invariant state is only possible in the so-called critical regime, i.e. there is a certain balance between birth and death. As it was shown in \cite{KKP, KKPZ}, in the case of the critical regime, there exists continuum of invariant measures parametrized by the density values. These invariant measures are described by a simple recurrent relation between their correlation functions and create a concrete (and up to our knowledge, completely new) class of random point fields. For all other regimes, the density of the system tends either to $\infty$ or to $0$ as time grows. The existence of invariant measures in the marked contact model in $\X$ with a compact spin space describing dynamics of a population with mutations was proved in \cite{KPZh}.

The goal of this work is to study the contact process in the critical regime running on general spaces. Our approach is based on the analysis of the infinite system of hierarchical equations for correlation functions, that has been studied earlier for the contact process in $\X$, see e.g. \cite{FKO, KKP}. We discuss these constructions and present the main result in Section 3. In Section 2, we formulate assumptions on the model that imply the existence of invariant measures for the contact processes running on general state spaces. In particular, our approach can be applied to the contact processes on a hyperbolic (Lobachevsky) space, on a Cayley tree as well as to the contact process on $\Z^d$ in inhomogeneous and random environments, see Section 4. In Section 4 we also present all  known results concerning invariant measures of the contact process in $\X$, $d \ge 1$.  Finally, Section 5 contains the proof of the main results.

\section{The model}

Let $\XXX$ be a locally compact separable metric space, ${\B}({\XXX})$ be its Borel $\sigma$-algebra, and $m$ will denote a locally finite  Borel measure on ${\B} (\XXX)$, i.e. $m$ is finite on compact sets.
Denote by ${\B}_{\mathrm{b}} ({\XXX})$ the system of all compact
sets from ${\B}({\XXX})$.
The continuous contact model is regarded as a  Markov process on $\XXX$ which is a
particular case of the general birth-and-death process.

By analogy with the continuous contact model in $\X$, see e.g. \cite{KKP, KS}, the phase space of a general contact process is the space of locally finite configurations in $\XXX$. Namely,
\begin{equation} \label{confspace}
\Ga =\Ga\bigl(\XXX\bigr) :=\Bigl\{ \ga \subset \XXX \Bigm| |\ga\cap\La
|<\infty, \ \mathrm{for \ all } \ \La \in {\B}_{\mathrm{b}}
(\XXX)\Bigr\}.
\end{equation}
Here $|\cdot|$ denotes the number of elements of a~set. We can identify each
$\ga\in\Ga$ with the non-negative Radon measure $\sum_{x\in
\gamma }\delta_x\in \M(\XXX)$, where $\delta_x$ is
the Dirac measure with unit mass at $x$, $\sum_{x\in\emptyset}\delta_x$ is, by definition, the zero measure, and $\M(\XXX)$ denotes the space of all
non-negative Radon measures on $\B(\XXX)$. This identification allows us to endow
$\Ga$ with the topology induced by the vague topology on
$\M(\XXX)$, i.e. the weakest topology on $\Ga$
with respect to which all the mappings
\begin{equation}\label{gentop}
    \Ga\ni\ga\mapsto \sum_{x\in\ga} f(x)\in{\R}
\end{equation}
are continuous for any $f\in C_0(\XXX)$ that is the set of all continuous functions on $\XXX$ with compact supports.
The topological space $\Gamma(X)$ for any $X\in \CB(\XXX)$ can be defined in a similar way. The Borel $\sigma$-algebra on $\Gamma(X)$ is denoted by $\CB(\Gamma(X))$.

The contact model is given by a heuristic
generator defined on a proper class of functions  $F:
\Gamma \to \R$ as follows:
\begin{align}
 \begin{aligned}\label{generator}
(L F)(\gamma) &= \sum_{ x \in \gamma}\left[F(\gamma \backslash \{x\}) - F(
\gamma)\right] \\
&+   \int\limits_{\XXX} \sum_{x \in \gamma}  a(y,x) (F(\gamma \,\cup
\{y\} ) - F(\gamma)) m(dy).
 \end{aligned}
 \end{align}
In the sequel, for simplicity of notations, we just write $x$
instead of $\{x\}$. The first term in \eqref{generator} corresponds to the death of the particles. Namely, each $x$ of the configuration $\gamma\in\Ga$ dies with the death rate $1$. The second term of \eqref{generator} describes the birth of a new particle in a small neighborhood $dy$ of the point $y$ with the
 birth rate density $b(y,\gamma):=\sum_{x \in \gamma} a(y,x)$.
\medskip

We assume that  $a:\XXX \times \XXX \to [0, \infty)$ is a non-negative bounded measurable function satisfying the following conditions:
\begin{enumerate}
\item {\it Critical regime} condition:
\begin{equation}\label{2a}
 \int\limits_{\XXX} a(x,y) m(dy) \ = \ 1 \quad \mbox{for all } x \in \XXX;
\end{equation}
\item {\it Transience} condition.
Let us consider the jump Markov process (random walk in continuum) with generator
\begin{equation}\label{L}
\mathcal{L} f(x) =  \int\limits_{\XXX} a(x,y) \big( f(y) - f(x) \big)  m(dy).
\end{equation}
Then we assume that for any two independent copies $X(t)$ and $Y(t)$ of this process starting with $X(0)=x$ and $Y(0)=y$ the following condition holds
\begin{equation}\label{2b}
\sup\limits_{x,y} \ \int\limits_{0}^{\infty} \mathbb{E}_{x,y} a(X(t), Y(t)) dt < Q
\end{equation}
with a constant $Q>0$. Moreover, we assume that the integral in \eqref{2b} converges uniformly in $x$, $y$.
\end{enumerate}
\medskip

\begin{remark}
The sufficient condition for \eqref{2b} together with required uniform convergence reads
\begin{equation}\label{suf-cond}
\int\limits_0^{\infty}  \sup\limits_{x, y} \ \mathbb{E}_x a(X(t), y) dt < Q.
\end{equation}
\end{remark}

\begin{proof} Denote by  $p (x,dy,t)$ the transition function of the Markov jump process with generator \eqref{L} at time $t$. Then we get
$$
\sup\limits_{x,y} \ \int\limits_{0}^{\infty} \mathbb{E}_{x,y} a(X(t), Y(t)) dt =
\sup\limits_{x,y} \ \int\limits_0^{\infty} \int\limits_{\XXX} \int\limits_{\XXX} a(x', y') p(x, dx',t) p(y, dy',t) dt \le
$$
$$
\sup\limits_{y} \ \int\limits_0^{\infty} \int\limits_{\XXX}  \Big( \sup\limits_{x} \  \int\limits_{\XXX}  a(x', y') p(x, dx', t) \Big)  p(y, dy',t)  dt =
$$
$$
\sup\limits_{y} \ \int\limits_0^{\infty}  \int\limits_{\XXX}  \ \Big( \sup\limits_{x} \ \mathbb{E}_x a(X(t), y')  \Big)  p(y, dy',t) dt \le
$$
$$
\int\limits_0^{\infty} \sup\limits_{y} \ \int\limits_{\XXX}  \Big( \sup\limits_{y'} \sup\limits_{x} \ \mathbb{E}_x a(X(t), y') \Big)  p(y, dy',t) dt =
\int\limits_0^{\infty}  \sup\limits_{x, y'} \ \mathbb{E}_x a(X(t), y') dt.
$$
Therefore, condition \eqref{suf-cond} implies the uniform convergence in \eqref{2b}.
\end{proof}

\section{Time evolution of correlation functions. Main results}

The study of evolution of the infinite-particle system generated by the operator \eqref{generator} may be realized through
the {\em forward Kolmogorov} (or {\em Fokker--Planck}) equation with the evolution operator $L$ for probability measures (states) on the configuration space $\Gamma$, i.e.
\begin{equation}\label{FPE-init}
  \frac{d}{d t} \mu_t(F) = \mu_t(LF),
  \quad t>0, \quad \mu_t\bigr|_{t=0}=\mu_0,
\end{equation}
where $$\mu(F):=\int_\Ga F(\ga)\,d\mu(\ga)$$.

Denote by ${\cal M}_{fm}(\Gamma)$ the set of all probability
measures $\mu$ which have finite local moments of all orders, i.e.
$$
\int_{\Gamma} |\gamma_{\Lambda} |^{n} \ \mu (d \gamma) \ < \ \infty
$$
for  all $\Lambda \in {\cal B}_b(\XXX)$ and $n \in N$, and let  ${\cal M}_{\rm{corr}}(\Gamma)$  be the subclass of ${\cal M}_{fm}(\Gamma)$ consisting of those probability measures on $\Ga$ for which  correlation functions (correlation measures densities) exist.  The terminology originates in statistical mechanics (see, for instance, \cite[Ch. 4]{R}) where the local densities of the corresponding correlation measure $\rho_\mu$ w.r.t. the Lebesgue measure in $\XXX^n = \mathbb{R}^{dn}$ are called correlation functions. For the correlation functions in the general space $\XXX$, see e.g. \cite{L1, L2}.

The evolution equation for the system of $n$-point correlation functions corresponding to the continuous contact model in $\X$ has been derived previously in \cite{KKP, KKPZ}.
This equation for the general contact model in
$\XXX$ can be considered in the same way. The equation has the following recurrent forms:
\begin{equation}\label{59}
\frac{\partial k_{t}^{(n)}}{\partial t} \ = \ \hat L_n^{\ast} k_{t}^{(n)} \
+ \ f_{t}^{(n)}, \quad n\ge 1; \qquad k_{t}^{(0)} \equiv 1,
\end{equation}
where
\begin{align}
 \begin{aligned}\label{korf}
 \hat L^{\ast}_n k^{(n)}(x_1, &\ldots, x_n)  =  - n \,k^{(n)}(x_1,
\ldots, x_n) \\
& + \sum_{i=1}^n \int\limits_{\XXX} a(x_i, y) k^{(n)}(x_1, \ldots, x_{i-1}, y,
x_{i+1}, \ldots, x_n) m(dy).
\end{aligned}
\end{align}
Here $f_{t}^{(n)}$ are functions on $\XXX^{n}$ defined for $n \ge 2$ by
\begin{equation}\label{f}
f_{t}^{(n)}(x_1, \ldots, x_n) \ = \ \sum_{i=1}^n  k_{t}^{(n-1)}(x_1,
\ldots,\check{x_i}, \ldots, x_n) \sum_{j\neq i} a(x_i, x_j),
\end{equation}
and  $f_{t}^{(1)} \equiv 0$.

We consider here the initial data $k_0 = \{k_0^{(n)} \}$ corresponding to the Poisson measure $\pi_\varrho$ with intensity $\varrho$:
\begin{equation}\label{k0}
k_0^{(0)}= 1, \quad k_0^{(n)}(x_1, \ldots, x_n) = \varrho^n,  \; n\ge 1.
\end{equation}

Let $\XX_{n} = \BB({\XXX}^n)$ be the Banach space of all measurable real-valued bounded functions on $\XXX^n$ with the $\sup$-norm.
Consider the operator $\hat L_n^{\ast}$ as an operator on the Banach space $\XX_{n}$ for any $n\geq 1$. Then it is a bounded linear operator in $\XX_{n}$, and the arguments based on the variation of parameters formula yields that
\begin{equation}\label{61A}
k_{t}^{(n)} \ = \ e^{t \hat L_n^{\ast}} k_{0}^{(n)} \ + \  \int\limits_0^t e^{(t-s) \hat
L_n^{\ast}} f_s^{(n)} \ ds,
\end{equation}
where $f_s^{(n)}$ is expressed through $k_s^{(n-1)}$ by (\ref{f}).
Thus, the solution to the Cauchy problem \eqref{59} in  $\XX_{n}$ with arbitrary initial values $k_{0}^{(n)}\in\XX_{n}$ exists and is unique provided $f_{t}^{(n)}$ is constructed recurrently via the
solution to the same Cauchy problem \eqref{59} for $n-1$.



The goal of this paper is to prove the existence of a family of invariant measures of the general contact process in the critical regime.
These measures are described in terms of the corresponding correlation functions
$\{k^{(n)}\}_{n\geq 0}$ as solutions to the following system satisfying the Lenard positivity condition  (see (\ref{LP}) below):
\begin{equation}\label{Last}
\hat L^{\ast}_n k^{(n)} + f^{(n)}=0, \quad n \ge 1, \quad
k^{(0)}\equiv 1,
\end{equation}
where $\hat L_n^{\ast}, \, f^{(n)}$ were defined by \eqref{korf}-\eqref{f}.
In the sequel, we say that $k:\Ga_{0}\to \R$ solves the system \eqref{Last} in the Banach spaces $(\XX_{n})_{n\geq 1}$ if the corresponding $k^{(n)}\in\XX_{n}$, $n\geq 1$ solve \eqref{Last}.

The main result of the paper is the following theorem.

\begin{theorem}\label{mainth} {\it Assume that the contact process satisfies conditions (\ref{2a}), (\ref{2b}). Then the following assertions hold.

{(i)} For any positive constant $\varrho >0$ there exists a
unique probability measure $\mu^{\varrho} \in {\cal M}_{\rm{corr}}(\Gamma)$ on $\Ga$ such that its
correlation function $k_{\varrho}: \Ga_{0}\to\R_{+}$ solves (\ref{Last})  in the Banach spaces $(\XX_{n})_{n\geq 1}$, and  the corresponding system $\{k_{\varrho}^{(n)}\}_{n\geq 1}$ satisfies $k_\varrho^{(1)}\equiv \varrho$. Moreover, there exists a positive constant $D$ such that
\begin{equation}\label{estimate}
k^{(n)}_\varrho (x_1, \ldots, x_n) \ \le   D  Q^n (n!)^2
\qquad  \text{for all} \quad (x_1, \ldots, x_n)\in {\XXX}^{n},
\end{equation}
where $Q$ is the same constant as in  (\ref{2b}).

{(ii)} Let $\{k_{\varrho, t}^{(n)}\}_{n\geq 1}$ be the
solution to the Cauchy problem (\ref{59}) with initial value (\ref{k0}). Then
\begin{equation}\label{Th1-2}
\| k_{\varrho, t}^{(n)} \ - \ k_\varrho^{(n)} \|_{\XX_n} \ \to \ 0, \quad t \to \infty, \quad \forall n\geq 1.
\end{equation}
}
\end{theorem}

\medskip
The main strategy of the proof follows the same line as the proof in the case $\XXX=\X$, see \cite{KKP, KKPZ}. However, in the present paper we should modify some of steps of the previous proof for the general models.
These modifications involve using the condition \eqref{2b} instead of using the Fourier transform for homogeneous models in $\X$.

\medskip
\begin{remark}
We can include in our model a possibility to jump. The analogous model in $\X$ has been considered earlier in \cite{KKPZ, KKS}.
More precisely, let us consider the following heuristic generator
$L\ + \ L_J$, where $L$  was defined by (\ref{generator}),
\begin{equation}\label{LJ}
L_J F(\gamma)\ = \  \int\limits_{\XXX} \sum_{x\in \gamma}
J(y,x) \Big( F ((\gamma\setminus x) \cup y)-F(\gamma) \Big) \, m(dy).
\end{equation}
Suppose that the total jump rate $\int J(y,x) m(dy)$
is uniformly bounded in $x$:
\begin{equation}\label{2aJ-bis}
\sup\limits_x \ \int\limits_{\XXX} J(y,x) m( dy) \ < \ C.
\end{equation}
Then the criticality condition is
\begin{equation}\label{criJ}
\int\limits_{\XXX} \big(a(x,y)+J(x,y) - J(y,x)\big) \ m(dy) =1
\end{equation}
("birth"+"immigration"-"emigration"="mortality"). The operator ${\mathcal{L}}_J$ analogous to \eqref{L} then takes the form
\begin{equation}\label{LambdaJ}
{\mathcal{L}}_J f(x) =  \int\limits_{\XXX} \big( a(x,y) + J(x,y) \big) \big( f(y) - f(x) \big) m(dy),
\end{equation}
and "transience" condition \eqref{2b} can be written as
\begin{equation}\label{2bJ}
\sup\limits_{x,y} \ \int\limits_{0}^{\infty} \mathbb{E}_{x,y} a( \tilde X(t), \tilde Y(t)) dt < Q,
\end{equation}
where $a(x,y)$ is the same birth rate as above satisfying  \eqref{2a}, \eqref{2b}, while $\tilde X(t)$ and $\tilde Y(t)$ are two independent copies of the Markov process with generator ${\mathcal{L}}_J$ given by \eqref{LambdaJ}.
\end{remark}



\section{Examples}

We start this section with the homogeneous contact model in $\X$ that has been studied in \cite{KKP, KKPZ, KPZh}. Other examples are new.

1. {\bf The homogeneous contact model in $\XXX=\X$ generated by birth rates }.
The homogeneous contact model in $\X$ has been studied in papers \cite{KKP, KKPZ, KPZh}, where we have formulated the condition on $a(x-y)$ guaranteeing the existence of a family of invariant measures of the contact model in the critical regime in any dimension $d \ge 1$. Namely,
we assume that $a(\cdot)$ possesses the following properties:\\
- Boundedness and Normalization
\begin{equation}\label{a1-R}
a(x) \ge 0; \quad  a(x) \in L^{\infty}(\mathbb R^d) \cap L^1(\mathbb R^d), \qquad \int_{\mathbb R^d} a(x) dx =1
\end{equation}
- Regularity condition
\begin{equation}\label{2b-R}
 \hat a (p) \ := \ \int\limits_{\R^{d}} e^{-i(p,u)} a(u) du \in
L^1(\R^{d}),
\end{equation}
-  Existence of the second moment in dimensions $d \ge 3$
\begin{equation}\label{a2-R}
 \int_{\mathbb R^d} |x|^2 a(x) dx  < \infty;
\end{equation}
-  Heavy tail conditions in dimensions $d =1,\ 2$:
\begin{equation}\label{a2.2}
a(x) \ \sim \ \frac{1}{|x|^{\alpha+2}} \quad \mbox{as } \; |x| \to \infty, \; 0<\alpha<2, \quad (d=2),
\end{equation}
\begin{equation}\label{a2.1}
a(x) \ \sim \ \frac{1}{|x|^{\alpha+1}} \quad \mbox{as } \; |x| \to \infty, \; 0<\alpha<1, \quad (d=1).
\end{equation}
Then the statements of Theorem \ref{mainth} are true, see \cite{KKP, KKPZ}. In these cases, the fulfillment of the transience condition \eqref{2b} is verified using the Fourier transform.

We consider in  \cite{KPZh} a marked continuous contact model on $\XXX= \X~\times~S, d \ge 3$, where $S$ is a compact metric space. The birth rates $a(x,y)$ were defined as
\begin{equation}\label{quasi-a}
a(x,y) = \alpha (\tau(x) - \tau(y)) \, \mathcal{Q} (\sigma(x), \sigma(y)),
\end{equation}
where $\tau$ and $\sigma$ are projections of $\XXX$ on $\X$ and $S$
respectively, $\alpha(\cdot) \ge 0$ is a function on $\X$ satisfying conditions \eqref{a1-R} - \eqref{a2-R} (the case $d \ge 3$).  We suppose that the function $\mathcal{Q}$ on $S \times S$ is continuous (and so bounded) and strictly positive. Moreover we assume that the corresponding integral operator with kernel $\mathcal{Q}(\cdot, \cdot)$ has the maximal in absolute value eigenvalue equal to 1. Then the statements of Theorem \ref{mainth} are true, see  \cite{KPZh}.
\medskip

2. {\bf The symmetric contact model on the hyperbolic (Lobachevsky) plane: $\XXX=L$.}
Let  $\rho(x,y)$ be the hyperbolic distance and $m(dx)$ be the corresponding measure on $\XXX$. Consider a continuous time random walk on $\XXX$ with the generator \eqref{L}, where $a(x,y)$ depends only on  $\rho(x,y)$:
\begin{equation}\label{rho-lob}
a(x,y) = a(\rho(x,y)) = a(y,x), \quad \mbox{ and } \;  a(x,y)=0 \quad \mbox{if } \; \rho(x,y)> h,
\end{equation}
for some $h$. We assume also that
\begin{equation}\label{2a-lob}
 \int\limits_{\XXX} a(x,y) m(dy) \  = \  1 \quad \mbox{for all } x \in \XXX.
\end{equation}
Then \eqref{2a-lob} is the same as the critical regime condition  \eqref{2a}, and it remains to check the fulfillment of the transience condition \eqref{2b}, or equivalently condition \eqref{suf-cond}.

Denote by $D(y,h)$ a disc centered at $y$ of radius $h$:
$$
D(y,h) = \{ y' \in \XXX: \ \rho(y,y') \le h \},
$$
and let $P(x, D(y,h), t)$ be the probability for the Markov jump process $X(t)$ with generator $L$ defined in \eqref{L} and starting at $x \in \XXX$ to reach $D(y,h)$ at time $t$:
$$
P(x, D(y,h), t) = \Pr \big( X(t) \in D(y,h)| \ X(0) =x \big).
$$

\begin{lemma}\label{Lemma-Lob}
There exist $\varkappa>0$ and $C(h)$ such that the probability $P(x, D(y,h), t)$ satisfies the following estimate
\begin{equation}\label{Lob-lemma}
P(x, D(y,h), t) \le C(h) e^{-\varkappa \, t} \quad \mbox{for all } \; x, \ y \in \XXX.
\end{equation}
\end{lemma}

\begin{proof}
Using the estimates from \cite{Tutu} (in the proof of \cite[Lemma 1]{Tutu}) we conclude that there exist $\alpha>0$ and $\gamma>0$ such that
\begin{equation}\label{Lob-1}
P(x, D(x, \alpha t), t) \le C_1(h) e^{- \gamma  t} \quad \mbox{for all } \; t \ge 0.
\end{equation}
Fix a large $t>0$. If $D(y,h) \subset D(x, \alpha t)$, then we get
\begin{equation}\label{Lob-2}
P(x, D(y, h), t) \le C_1(h) e^{- \gamma  t}.
\end{equation}
Let us consider the case when $\rho(x,y)> \alpha t -h $  for a given $t$, where $\alpha>0$ is the same constant as in \eqref{Lob-1}.
In this case, using that the length of the circle of a radius $r$ is exponentially large in $r$ we conclude that "the visible angular size from $x$", i.e. the
angle $\varphi(y,h)$ of a sector centered at $x$ and resting on disk $D(y,h)$ admits the following upper bound
$$
\varphi(y,h) \le \tilde C_2(h) e^{- \alpha t}.
$$
Then isotropy condition \eqref{rho-lob} implies that
\begin{equation}\label{Lob-3}
P(x, D(y, h), t) \le  C_2 (h)  e^{- \alpha t}\quad \mbox{ if  } \; \rho(x,y)> \alpha t -h.
\end{equation}
Taking
$$
\varkappa = \min \{ \alpha, \ \gamma \}, \quad C(h) = \max \{ C_1(h), \ C_2(h) \},
$$
we obtain desired estimate \eqref{Lob-lemma} from \eqref{Lob-2} - \eqref{Lob-3}.
\end{proof}
This Lemma immediately implies the convergence of the integral in  \eqref{suf-cond}, since
$$
 \sup\limits_{x,y} \mathbb{E}_{x} a(X(t), y) \le A  \sup\limits_{x,y} P(x, D(y,h), t) \le A  C(h) e^{-\varkappa \, t},
$$
where $A = \sup a(x,y)$.

The symmetric contact model on the hyperbolic space ($d \ge 3 $) can be considered in the similar way.
\medskip

3. {\bf The homogeneous contact model on the Cayley tree: $\XXX=T_k$.}\\
Consider the Cayley tree $T_k$, i.e. infinite regular tree with vertex degree $k\ge 3$. The continuous time symmetric random walk $X(t)$ on $T_k$ is given by the generator $L$ defined in \eqref{L}, where $a(x,y)= a(y,x) =\frac1k$ for $d(x,y)=1$ and $a(x,y) =0$ otherwise. Here $d(x,y)$ is the distance on $T_k$, which is the same as the distance between two vertices in a graph. The measure $m$ on $T_k$ is the counting measure.

It is easy to see that for the trajectory of random walk $X(t)$ starting at $X(0) =x$ and any vertex $y \in T_k$ the distance $d_{x,y} (t) = d(X(t),y)$ is a random walk $Z_{x,y} (t)$ on $\mathbb{Z}_+ = \{ 0, 1, 2,\ldots \}$ with rates
$r(0,1)=1$ and $r(n,n+1) = \frac{k-1}{k}, \ r(n,n-1) = \frac1k, \ n=1,2, \ldots$. Thus, the random walk $Z_{x,y}(t)$ for any $x$ and $y$ has a positive drift, and the following lemma holds.
\begin{lemma}\label{Lemma-tree}
There exist $\varkappa>0$ and $C$ such that the transition probability $P(x, y, t) = \Pr ( X(t)=y | \, X(0) = x)$ meets the following estimate:
\begin{equation}\label{tree-lemma}
P(x, y, t) \le C e^{-\varkappa \, t} \quad \mbox{for all } \; x, \ y \in T_k.
\end{equation}
\end{lemma}
\begin{proof}
The proof of this lemma is completely analogous to the proof of Lemma \ref{Lemma-Lob}. As above, it is a  combination of two estimates. The first bound follows from the fact that the random walk $Z_{x,y}(t)$ for any $x$ and $y$ has a positive drift. Namely, there exist $\alpha>0$ and $\gamma>0$ such that
\begin{equation}\label{T-1}
P(x, D(x, \alpha t), t) = \Pr (Z_{x,x}(t) < \alpha t ) \le C_1 e^{- \gamma  t} \quad \mbox{for all } \; t \ge 0,
\end{equation}
where $D(x,R)$ is the ball centered at $x$ of radius $R$.
The second estimate is valid for $y \in T_k$ with  $d(x,y)> \alpha t$:
\begin{equation}\label{T-3}
P(x, y, t) \le  C_2  e^{- \alpha t} \quad \mbox{for some } \; \alpha>0.
\end{equation}
Estimate \eqref{T-3} follows from the observation that the number of $y \in T_k$ such that $d(x,y) =n$ is exponential in $n$. Finally, \eqref{tree-lemma} follows from \eqref{T-1} - \eqref{T-3}.
\end{proof}
Thus, as above we conclude that all required conditions \eqref{2a} and \eqref{2b} on $a(x,y)$ are fulfilled.

Let us note that the analogous result holds for any tree $T$ with vertex degree $k_i \ge 3, \ i \in T$.
\medskip

4. {\bf Inhomogeneous contact models generated by inhomogeneous random walks on a lattice: $\XXX=\Z^d$.}
Let us consider a random walk on $\Z^d$ with transition probabilities $P(x,y) = \Pr (x \to y)$ that differ from those of the homogeneous symmetric walk $\pi(y-x) = \pi(x-y)$ only locally, i.e. in a finite neighborhood of the origin:
\begin{equation}\label{inhomo}
P (x, y) = \pi (y-x) + V(y-x, x),\qquad \sum_{y} P (x, y) = 1 \;\; \forall \; x \in \Z^d,
\end{equation}
with
$$
V(u,x)=0 \; \mbox{ if } \; \max\{|u|, \ |x|  \} > R
$$
for some $R>0$. Moreover, assume that the perturbed random walk is irreducible in $\Z^d$.

In the paper \cite{MZh}, the main term of the asymptotics of the probability ${\Pr} (X(t) = y |\, X(0) = x)$ as $t \to \infty $ has been found. It turns out that, for $d \ge 2$, this main term of the asymptotics differs from the corresponding term of the asymptotics for the homogeneous  symmetric walk (which has a usual Gaussian form) by a quantity of the order $O(t^{- d/2} (|y| + 1)^{ - \frac{(d - 1)}{2}} )$. Thus the correction to the Gaussian term is comparable with it only in a finite neighborhood of the origin, and consequently the following uniform in $x$  and  $y$ estimate holds:
\begin{equation}\label{nhrw}
p(x,y,t) := \Pr (X(t) = y|\, X(0) = x) \le \frac{C}{t^{d/2}}
\end{equation}
with a constant $C$ depending on probabilities \eqref{inhomo} and $d$.

Taking into account critical regime condition  \eqref{2a} one can conclude that estimate \eqref{nhrw} also holds for the continuous time random walk with jump intensities $a(x,y) = P (x, y)$.  Thus, estimate \eqref{nhrw} implies that in the case $d \ge 3$ the transience condition  \eqref{2b} is fulfilled. Consequently, all statements of Theorem \ref{mainth} are valid for the inhomogeneous contact process on $\Z^d$ with the birth rates given by $a(x,y) = P(x,y)$.
\medskip

5. {\bf  Contact models on $\XXX=\Z^d$ generated by random conductance models.}
Let $\mathbb{E}_d$ be the set of non-oriented nearest neighbour bonds of the lattice $\Z^d$:
$$
\mathbb{E}_d = \{ (x,y) \in \Z^d \times \Z^d, \ x \sim y\}, \quad x \sim y \; \mbox{means} \; x,y \; \mbox{ are neighbors},
$$
and $\mu_e, \ e \in \mathbb{E}_d,$ are taken as nonnegative i.i.d.r.v defined on a probability space $(\Omega, \mathbb{P})$. Moreover, assume that
\begin{equation}\label{ellip}
 c^{-1} \le \mu_e \le c \quad \mbox{for some } \;  c \ge 1.
\end{equation}
Thus, $\mu_{xy} = \mu_{yx}, \ x\sim y$, are i.i.d. random variables satisfying \eqref{ellip}, and $\mu_{xy} = 0$ if $x \not \sim y$. Set
$$
\mu_x = \sum_y \mu_{xy}, \quad a(x,y) = \frac{\mu_{xy}}{\mu_x},
$$
and consider a continuous time random walk on $\Z^d$ with transition rates $a(x,y)$. The generator of this random walk is given by
$$
{\mathcal{L}}_C f(x) = \mu_x^{-1} \sum_y \mu_{xy} (f(y) - f(x)).
$$
In this case, see e.g. \cite{BD}, \cite{DD}, the following upper bound holds $\mathbb{P}$-a.s.
$$
\sup_{x,y} \, p^\omega(x,y,t)  := \sup_{x,y}\, {\Pr}^\omega (X(t) = y | \, X(0) = x) \le \frac{C}{t^{d/2}},
$$
where constant $C$ does not depend on $\omega$.

Moreover, the same result holds for the simple random walk on the infinite Bernoulli (bond) percolation cluster in $\Z^d$, see e.g. \cite{BH, MR}.

Thus, in both models the transience condition \eqref{suf-cond} (and  \eqref{2b}) guaranteeing the existence of the invariant measure of the corresponding contact model is fulfilled  in the case $d \ge 3$.

\section{The proof of Theorem \ref{mainth}}


In the proof of the first part of Theorem \ref{mainth} we use the
induction in $n\in\N$.
For $n=1$ in (\ref{Last}) we have
\begin{equation}\label{8}
-k^{(1)}(x) + \int\limits_{\XXX} a(x,y) k^{(1)}(y) m(dy) = 0.
\end{equation}
It follows immediately that $k^{(1)}(x) \equiv \varrho $ is an element of $\XX_{1}$ and it solves \eqref{8}.
We notice that $\varrho$ can be interpreted as the spatial density of particles.

Now let us turn to the general case. If for any $n>1$ we succeed to
solve equation (\ref{Last}) and express $k^{(n)}$ through
$f^{(n)}$, then knowing the expression of $f^{(n)}$ through
$k^{(n-1)}$ (see (\ref{f})), we get the solution $\{k^{(n)}\}_{n\geq 1}$ to the full system
(\ref{Last}) recurrently.
%

\begin{lemma} \label{3.2}
The operator $e^{t \hat L_n^{\ast}}$, where $\hat L_n^{\ast}$ was defined in \eqref{korf}, is
positive, i.e. it maps non-negative functions to non-negative functions.
\end{lemma}
\begin{proof}
The operator
$$
A^i k^{(n)}(x_1, \ldots, x_n) \ := \ \int\limits_{\XXX} a(x_i, y) k^{(n)} (x_1, \ldots, x_{i-1}, y,
x_{i+1}, \ldots, x_n) m(dy).
$$
is positive and bounded on $\XX_{n}$ for any $1\leq i\leq n$. Taking into account
\begin{equation}\label{20L}
e^{t \hat L_n^{\ast}} \ = \ \otimes_{i=1}^n \ e^{-t} e^{t A^{i}},
\end{equation}
we get the desired conclusion.
\end{proof}

Next we will construct a solution to the system (\ref{Last}) satisfying \eqref{estimate}.
As follows from (\ref{f}), the function $f^{(n)}$ is the sum of
functions of the form
\begin{equation}\label{32}
f_{i,j} (x_1, \ldots, x_n)  =  k^{(n-1)} (x_1,\ldots,\check{x_i},
\ldots, x_n)  a(x_i, x_j), \quad i\neq j.
\end{equation}

We suppose by induction that
$$
k^{(n-1)} (x_1, \ldots, x_{n-1}) \ \le \ K_{n-1}, \quad \text{for all } \; (x_1, \ldots, x_{n-1})\in {\XXX}^{n-1},\quad n\geq 2,
$$
where $K_n = D C^n (n!)^2$, and $D, C$ are some constants. Consequently,
\begin{equation}\label{34}
f_{i,j}(x_1, \ldots, x_n) \ \le \  K_{n-1} a(x_i, x_j),\quad (x_1, \ldots, x_{n})\in {\XXX}^{n}.
\end{equation}
Using the positivity of the operator $e^{t \hat L_n^{\ast}}$ and (\ref{34}) we have
\begin{align}
\begin{aligned}\label{36}
\left(e^{t \hat L_n^{\ast}} f_{i,j} \right) (x_1, \ldots, x_{n})\ \le \  K_{n-1} \ \left(e^{t \hat L_n^{\ast}}
a(\cdot_i, \cdot_j) \right)(x_1, \ldots, x_{n}).
\end{aligned}
\end{align}
Set
\begin{align}
\begin{aligned}\label{21}
{\mathcal{L}}^{i} k^{(n)}(x_1, \ldots, x_n) \ = &\int\limits_{\XXX} a(x_i, y) k^{(n)}(x_1, \ldots, x_{i-1}, y, x_{i+1}, \ldots,
x_n) m( dy)\\
& - k^{(n)}(x_1, \ldots, x_n).
\end{aligned}
\end{align}
An easy observation $e^{t {\mathcal{L}}^{i}}1\!\!1=1\!\!1$, $\forall i=1,\,\ldots, n,$ where
 $1\!\!1(x)\equiv 1$, shows
\begin{equation}\label{36_1}
\left(e^{t \hat L_n^{\ast}}
a(\cdot_i, \cdot_j)\right) (x_1, \ldots, x_{n})\  =  \
\left(e^{t ( {\mathcal{L}}^i + {\mathcal{L}}^j)} a(\cdot_i, \cdot_j)\right) (x_1, \ldots, x_{n}).
\end{equation}
Note that the latter function depends only on variables $x_{i}$ and $x_{j}$.
\medskip

Notice that $e^{t \hat L_n^{\ast}} f_{i,j}$ is integrable with respect to $t$ on $\R_{+}$. It immediately follows from  \eqref{20L}, \eqref{32} and condition \eqref{2b}:
\begin{equation}\label{vij}
v^{(n)}_{i,j} \ = \ \int_0^{\infty} e^{t \hat L_n^{\ast}} f_{i,j} \ dt \
\le  K_{n-1} \, Q,
\end{equation}
where $Q$ is the same constant as in \eqref{2b}.
We used here the identity
\begin{equation}\label{2BB}
 e^{t \hat L_n^{\ast}} a(x,y)   =   \mathbb{E}_{x,y} a(X(t), Y(t)).
\end{equation}

Our next goal is to show that
\begin{equation}\label{k-def}
v^{(n)}  = \sum_{i \neq j} v^{(n)}_{i,j}=\int_0^{\infty} e^{t \hat L_n^{\ast}} f^{(n)} dt
\end{equation}
is a solution to \eqref{Last} in $\XX_{n}$.
It is easily seen from \eqref{vij} and induction procedure that $v^{(n)}\in\XX_{n}$. Since $e^{t \hat L_n^{\ast}}$ is a strongly continuous semigroup we have
\begin{equation}\label{37A}
e^{t \hat L_n^{\ast}}f^{(n)}-f^{(n)}=\hat L_n^{\ast}\int_{0}^{t}e^{s \hat L_n^{\ast}}f^{(n)}ds.
\end{equation}
Rewrite \eqref{37A} as
\begin{equation}\label{37A-bis}
e^{t \hat L_n^{\ast}}f^{(n)} = f^{(n)}+ \hat L_n^{\ast}\int_{0}^{t}e^{s \hat L_n^{\ast}}f^{(n)}ds.
\end{equation}
Then using condition \eqref{2b}, inequality \eqref{34}, Lemma \ref{3.2} and the fact that $\hat L_n^{\ast}$ is a bounded operator we conclude that the right hand side of \eqref{37A-bis} has a uniform in $x_1, \ldots, x_n$ limit as $t \to \infty$, therefore, the left hand side of \eqref{37A-bis}, i.e. $ e^{t \hat L_n^{\ast}}f^{(n)}$, also converges in $\XX_{n}$. Moreover the limit is a nonnegative function in  $\XX_{n}$. However, if this function is somewhere strictly positive, then we get a contradiction with condition \eqref{2b}. Thus, we conclude that the following limit holds in $\XX_{n}$:
\begin{equation}\label{2c-bis}
 e^{t \hat L_n^{\ast}}f^{(n)} \to 0, \quad t\to \infty.
\end{equation}
A passage to the limit in \eqref{37A}  as $t\to\infty$ together with \eqref{2c-bis} shows  that $v^{(n)}$ defined in \eqref{k-def} is a solution to \eqref{Last} in $\XX_{n}$.

Since the function $f^{(n)}$ is the sum
of functions $f_{i,j}$, $i\neq j$ we deduce from \eqref{vij} that $ v^{(n)}$
is bounded by $n^2  K_{n-1} Q$. Thus we get the recurrence inequality
\begin{equation}\label{50}
K_n \ \le \ n^2 K_{n-1} Q,
\end{equation}
and by induction it follows that
\begin{equation}\label{49}
K_n \ \le \ Q^n \, (n!)^2.
\end{equation}
Thus
\begin{equation}\label{49A}
v^{(n)} (x_1, \ldots, x_n) \ \le \ Q^n \, (n!)^2.
\end{equation}

Thus, we have constructed $\{v^{(n)}\}_{n\geq 1}$ satisfying estimate (\ref{49A}).
Of course, any functions of the form
$$
k^{(1)}\equiv \varrho,\quad k^{(n)}  = v^{(n)} + A_n =  \int\limits_0^{\infty} e^{t \hat
L_n^{\ast}} f^{(n)} \ dt  + A_n,\quad n \geq 2,
$$
where $A_n$ are arbitrary constants, are solution to the system
(\ref{Last}) too.  Taking $A_n=\varrho^n$ we conclude that
\begin{equation}\label{52}
k^{(1)}_\varrho\equiv\varrho,\quad k^{(n)}_\varrho = v^{(n)} + \varrho^n = \int\limits_0^{\infty} e^{t \hat L_n^{\ast}} f^{(n)}dt  +  \varrho^n,\quad n\geq 2,
\end{equation}
is the desired solution to \eqref{Last} in the Banach spaces $(\XX_{n})_{n\geq 1}$. To emphasize the dependence of $f^{(n)}$ on $\varrho$, we will use notation $f_{\varrho}^{(n)}$  for $f^{(n)}$.
For the solutions $\{ k_{\varrho}^{(n)} \}_{n\geq 1}$ of \eqref{52} instead of (\ref{50}) we have the recurrence
\begin{equation}\label{53}
K_n \ \le \  n^2 K_{n-1} Q \ + \ \varrho^n,
\end{equation}
which yields
\begin{equation}\label{55}
K_n \ \le \ D Q^n (n!)^2.
\end{equation}

To be certain that the constructed system $\{ k_{\varrho}^{(n)} \}_{n\geq 1}$ is a system of correlation functions, i.e., it corresponds to a probability measure $\mu^\varrho$ on the configuration space $\Gamma$, we will
prove below that $\{ k_{\varrho}^{(n)} \}_{n\geq 1}$ can be constructed as the limit when $t \to \infty$ of the system
of correlation functions $\{ k_{t}^{(n)}\}_{n\geq 1}$ associated with the solution to the Cauchy problem \eqref{59} with
the initial data \eqref{k0}.


By the variation of parameters formula we have
\begin{equation}\label{61}
k_{t}^{(n)} \ = \ e^{t \hat L_n^{\ast}} k_{0}^{(n)} \ + \  \int\limits_0^t e^{(t-s) \hat
L_n^{\ast}} f_s^{(n)} \ ds,
\end{equation}
where $f_s^{(n)}$ is expressed through $k_s^{(n-1)}$ by
(\ref{f}). On the other hand, we proved above the existence of the solution $\{ k_{\varrho}^{(n)} \}_{n\geq 1}$ of the stationary problem:
$$ \hat L_n^{\ast}
k_\varrho^{(n)} \ = \ - f_\varrho^{(n)},$$
where
$$
f_\varrho^{(n)}(x_1, \ldots, x_n) \ = \ \sum_{i,j:\ i\neq j}
k_\varrho^{(n-1)}(x_1, \ldots,\check{x_i}, \ldots, x_n) \ a(x_i, x_j).
$$
This solution meets the following equation
$$
\left( e^{t \hat L_n^{\ast}} - E \right) k_\varrho^{(n)} \ = \ -
\int\limits_0^t \frac{d}{ds} e^{(t-s) \hat L_n^{\ast}} k_\varrho^{(n)} ds \
\ = \ - \int\limits_0^t e^{(t-s) \hat L_n^{\ast}}  f_\varrho^{(n)} \ ds,
$$
and therefore
\begin{equation}\label{64}
k_{t}^{(n)} -  k_\varrho^{(n)} \ = \
e^{t \hat L_n^{\ast}}(k_{0}^{(n)} - k_\varrho^{(n)}) \ + \  \int\limits_0^t
e^{(t-s) \hat L_n^{\ast}} (f_{s}^{(n)} -  f_\varrho^{(n)}) \ ds.
\end{equation}
We will prove now that both terms in the right-hand side of (\ref{64}) converge to 0 in
the norm of $\XX_n$ as $t \to \infty$.

Formula (\ref{52}) yields
\begin{equation}\label{65}
e^{t \hat L_n^{\ast}} \big( k_{0}^{(n)} - k_\varrho^{(n)} \big) \ = \ - e^{t \hat
L_n^{\ast}} v^{(n)},
\end{equation}
where
\begin{equation}\label{66}
v^{(n)} \ = \ \int_0^{\infty} e^{s \hat L_n^{\ast}} f_\varrho^{(n)}
\ ds.
\end{equation}
Consequently, the first term in the r.h.s. of  \eqref{64} can be rewritten using \eqref{65} as follows
\begin{equation*}\label{1termA}
e^{t \hat L_n^{\ast}} \ v^{(n)} = \int_{0}^{\infty}e^{(t+s) \hat L_n^{\ast}}f_\varrho^{(n)} \, ds =  \int_{t}^{\infty}e^{r \hat L_n^{\ast}}f_\varrho^{(n)} \, dr.
\end{equation*}
Due to the uniform convergence of the integral in \eqref{2b} we conclude that
\begin{equation}\label{1term}
||e^{t \hat L_n^{\ast}} \ v^{(n)}||_{{\XX}_n}\to 0,\quad t\to\infty.
\end{equation}

The second term in the r.h.s. of   \eqref{64} can be estimated in the same way as in our previous works \cite{KKPZ, KPZh}.

Thus we proved the strong convergence (\ref{Th1-2}), and the proof of the second part of Theorem \ref{mainth} is completed.
\medskip

The final step of the proof is to show that the system of correlation functions $\{k_\varrho^{(n)} \}$ corresponds to a probability measure $\mu^\varrho$ on the configuration space $\Gamma$. For this  we have constructed above $\{k_\varrho^{(n)} \}$ as the limit when $t \to \infty$ of the solution $\{ k_{t}^{(n)} \}$
of the Cauchy problem (\ref{59}) with initial data (\ref{k0})
\begin{equation}\label{limk}
k^{(n)}_\varrho \ = \ \lim_{t\to\infty} k_{t}^{(n)}.
\end{equation}
We will use next the following Proposition summarizing results of two papers \cite{L1} and \cite{L2} of A. Lenard.

\begin{proposition}(see \cite{L1}, \cite{L2}) \label{propos2} If the system of correlation functions $\{k^{(n)} \}$ satisfies Lenard positivity and moment growth conditions then there exists a unique probability measure $\mu \in {\cal M}^1_{fm}(\Gamma)$, locally absolutely continuous with respect to a Poisson measure, whose correlation functions are exactly $\{k^{(n)} \}$.
\end{proposition}

For the convenience of the reader we formulate these conditions below.

\noindent {\it Lenard positivity.} $KG \ge 0$ for any $G \in B_{bs}(\Gamma_0)$ implies
\begin{equation}\label{LP}
\sum_{n=0}^{\infty} \frac{1}{n!} \int\limits_{\XXX} \ldots \int\limits_{\XXX} G^{(n)}(x_1, \ldots, x_n) k^{(n)}(x_1, \ldots, x_n) m(dx_1) \ldots m(dx_n) \ge 0.
\end{equation}
\medskip

\noindent {\it Moment growth.} For any bounded set $\Lambda \subset \XXX$ and $j \ge 0$
\begin{equation}\label{MG}
\sum_{n=0}^{\infty} (m^\Lambda_{n+j})^{- \frac{1}{n}} \ = \ \infty,
\end{equation}
where
$$
m^\Lambda_{n} = (n!)^{-1} \int\limits_\Lambda \ldots \int\limits_{\Lambda} k^{(n)}(x_1, \ldots, x_n) m(dx_1) \ldots m(dx_n).
$$
\medskip

In our case the inequality $\big( m_n^\Lambda \big)^{-\frac1n} \ge \frac{\tilde C}{n}$ follows from bound (\ref{55}). Thus condition (\ref{MG}) of the uniqueness holds.
To obtain the Lenard positivity condition (\ref{LP}) we use \eqref{limk}.
It follows from results of \cite{KKP} (Proposition 4.4 and Corollary 4.1) that  for any $t>0$ the solution $\{
k_{t}^{(n)} \}$ of the Cauchy problem (\ref{59}) satisfies condition (\ref{LP}) of Lenard positivity, see Appendix in \cite{KKPZ} for the detailed proof of this important statement.
Consequently, the limit system of correlation functions $k^{(n)}_\varrho $
also satisfies the Lenard positivity condition (\ref{LP}).

Thus Proposition \ref{propos2} implies that for any  $\varrho >0$  there exists a unique probability measure $\mu^\varrho \in {\cal M}^1_{\rm{corr}}(\Gamma)$, locally absolutely continuous with respect to a Poisson measure $\pi_\varrho$, and whose correlation functions are $\{k^{(n)}_\varrho \}$.
This completed the proof of Theorem \ref{mainth}.


\begin{thebibliography}{20}


\bibitem{BD} M.T. Barlow, J.-D. Deuschel, Invariance principle for the random conductance model with
unbounded conductances, Ann. Probab. 38, No. 1, 234-276 (2010)


\bibitem{BH}  M.T. Barlow, B.M. Hambly, Parabolic Harnack inequality and local limit theorem for percolation clusters,
Electron. J. Probab. 14, 1-26 (2009)


\bibitem{DD}  T. Delmotte, J.-D. Deuschel, On estimating the derivatives of symmetric diffusions in stationary random environment, with application to $\nabla \phi$ interface model, Probab. Theory Related Fields 133, 358-390 (2005)


\bibitem{FKO} D.L. Finkelshtein, Yu.G. Kondratiev, M.J. Oliveira, Markov evolutions and hierarchical equations in the continuum. I: one-comonent systems, Journal of Evolution Equations, 9, 197-233 (2009)

\bibitem{H} T.E. Harris, Contact interactions on a lattice, Ann. Probab. 2, 969-988 (1974)

\bibitem{HL} R. Holley, T.M. Liggett, The survival of contact processes, Ann. Probab. 6, No. 2, 198-206 (1978)

\bibitem{KK2002}
Kondratiev, Y., Kuna, T.: Harmonic analysis on configuration space.
{I}.
  {G}eneral theory.
\newblock Infin. Dimens. Anal. Quantum Probab. Relat. Top. \textbf{5}(2),
  201--233 (2002)



\bibitem{KKP} Yu. Kondratiev, O. Kutoviy, S. Pirogov, Correlation functions
and invariant measures in continuous contact model, Ininite
Dimensional Analysis, Quantum Probability and Related Topics Vol.
11, No. 2, 231-258 (2008)

\bibitem{KKPZ} Yu. Kondratiev, O. Kutoviy, S. Pirogov, E. Zhizhina, Invariant measures for spatial contact model in small dimensions, Arxiv: 1812.00795, 29 November 2018.

\bibitem{KKS} Yu. G. Kondratiev, O. V. Kutoviy, S. Struckmeier, Contact model with Kawasaki dynamics in continuum,
SFB-701 Preprint, University of Bielefeld, Bielefeld, Germany (2007).


\bibitem{KPZh} Yu. Kondratiev, S. Pirogov, E. Zhizhina, A Quasispecies Continuous Contact Model
in a Critical Regime, Journal of Statistical Physics, 163(2), 357-373 (2016), doi:10.1007/s10955-016-1480-5


\bibitem{KS} Yu. G. Kondratiev and A. Skorokhod, On contact processes in continuum,
Ininite Dimensional Analysis, Quantum Probability and Related Topics
Vol. 9, 187-198 (2006)



\bibitem{L1} A. Lenard, Correlation functions and the uniqueness of the state in classical statistical mechanics,
Comm. Math. Phys. 30, 35-44 (1973).

\bibitem{L2} A. Lenard, States of classical statistical mechanical systems of infinitely many particles II. Characterization of correlation measures, Arch. Rational Mech. Anal.
59, II: 240-256 (1975).

\bibitem{Lig1985}
T.~M. Liggett.
\newblock {\em Interacting particle systems}, volume 276 of {\em Grundlehren
  der Mathematischen Wissenschaften [Fundamental Principles of Mathematical
  Sciences]}.
\newblock Springer-Verlag, New York, 1985.


\bibitem{MR} P. Mathieu, E. Remy, Isoperimetry and heat kernel decay on percolation clusters, Ann. Probab. 32, 100-128 (2004)

\bibitem{MZh} R. A. Minlos, E. A. Zhizhina, A local limit theorem for nonhomogeneous random walk on a lattice, Theory Probab. Appl., Vol. 39 (3), 490-503 (1994).


\bibitem{R} D. Ruelle, Statistical Mechanics, Benjamin (1969).

\bibitem{Tutu} V. N. Tutubalin, On the limit behaviour of composition of measures in the plane and space of Lobachevski, Theor. Probab. Appl., Vol. 7, pp.189-196 (1962).

\end{thebibliography}
\end{document}